\documentclass[12pt]{amsart}
\usepackage{amscd,amsmath,amsthm,amssymb,graphics}
\usepackage{pstcol,pst-plot,pst-3d}%\usepackage[T1]{fontenc}
\usepackage{lmodern,pst-node}%\usepackage{geometric}
\usepackage{multicol}
\usepackage{amssymb}
\usepackage{graphicx}
\usepackage{tikz}
\usepackage{tkz-berge}
\usepackage{subfig}
\usepackage{hyperref}
\usepackage{mathtools}
\usepackage{tikz-3dplot}
\usepackage{enumitem}
\usepackage{graphicx}
\usepackage{tikz}
\usepackage{tkz-berge}
\usepackage{enumitem}
\usepackage{enumerate}
\usepackage{epic,eepic}
\usepackage{amsfonts,amssymb,amscd,amsmath,enumerate,verbatim}
\psset{unit=0.7cm,linewidth=0.8pt,arrowsize=2.5pt 4}
\newpsstyle{fatline}{linewidth=1.5pt}
\newpsstyle{fyp}{fillstyle=solid,fillcolor=verylight}
\definecolor{verylight}{gray}{0.97}
\definecolor{light}{gray}{0.9}
\definecolor{medium}{gray}{0.85}
\definecolor{dark}{gray}{0.6}

\def\frk{\frak}               % font for "Fraktur"

\def\Phi{{\frk n}}
\def\Phi{{\frk N}}
\def\opn#1#2{\def#1{\operatorname{#2}}} % to make operators
\opn\chara{char} \opn\length{\ell} \opn\pd{pd} \opn\rk{rk}
\opn\projdim{proj\,dim} \opn\injdim{inj\,dim} \opn\rank{rank}
\opn\depth{depth} \opn\grade{grade} \opn\height{height}
\opn\embdim{emb\,dim} \opn\codim{codim}

\opn\Tr{Tr} \opn\bigrank{big\,rank}
\opn\superheight{superheight}\opn\lcm{lcm}
\opn\trdeg{tr\,deg}%\emph{
\opn\reg{reg} \opn\lreg{lreg} \opn\ini{in} \opn\lpd{lpd}
\opn\size{size}\opn\bigsize{bigsize}
\opn\cosize{cosize}\opn\bigcosize{bigcosize}
\opn\sdepth{sdepth}\opn\sreg{sreg}
\opn\link{link}\opn\fdepth{fdepth} \opn\trdeg{trdeg} \opn\mod{mod}
\opn\div{div} \opn\Div{Div} \opn\cl{cl} \opn\Cl{Cl}
\opn\Spec{Spec} \opn\Supp{Supp} \opn\supp{supp} \opn\Sing{Sing}
\opn\Ass{Ass} \opn\Min{Min}\opn\Mon{Mon} \opn\dstab{dstab} \opn\astab{astab}
\opn\Syz{Syz}
\opn\Ann{Ann} \opn\Rad{Rad} \opn\Soc{Soc}
\opn\Im{Im} \opn\Ker{Ker} \opn\Coker{Coker} \opn\Am{Am}
\opn\Hom{Hom} \opn\Tor{Tor} \opn\Ext{Ext} \opn\End{End}
\opn\Aut{Aut} \opn\id{id}

\opn\nat{nat}
\opn\pff{pf}%   \pf exists already
\opn\Pf{Pf} \opn\GL{GL} \opn\SL{SL} \opn\mod{mod} \opn\ord{ord}
\opn\Gin{Gin} \opn\Hilb{Hilb}\opn\sort{sort}
\opn\S{S} \opn\dim{dim} \opn\supp{supp}\opn\trdeg{trdeg}\opn\sort{sort}
\opn\aff{aff} \opn\con{conv} \opn\relint{relint} \opn\st{st}
\opn\lk{lk} \opn\cn{cn} \opn\core{core} \opn\vol{vol}
\opn\link{link} \opn\star{star}\opn\lex{lex}
\opn\conv{conv} \opn\Ehr{Ehr}
\opn\gr{gr}

\def\pot#1#2{#1[\kern-0.28ex[#2]\kern-0.28ex]}

\opn\dirlim{\underrightarrow{\lim}}
\opn\inivlim{\underleftarrow{\lim}}
\def\Implies{\ifmmode\Longrightarrow \else
        \unskip${}\Longrightarrow{}$\ignorespaces\fi}
\def\implies{\ifmmode\Rightarrow \else
        \unskip${}\Rightarrow{}$\ignorespaces\fi}
\def\iff{\ifmmode\Longleftrightarrow \else
        \unskip${}\Longleftrightarrow{}$\ignorespaces\fi}

\let\:=\colon
\newtheorem{Theorem}{Theorem}[section]
 
 \newtheorem{Corollary}[Theorem]{Corollary}
 \newtheorem{Proposition}[Theorem]{Proposition}
 \newtheorem{Remark}[Theorem]{Remark}
 
 \newtheorem{Example}[Theorem]{Example}
 
 \newtheorem{Definition}[Theorem]{Definition}

\let\epsilon\varepsilon
\let\kappa=\varkappa
\def\qed{\ifhmode\textqed\fi
      \ifmmode\ifinner\quad\qedsymbol\else\dispqed\fi\fi}
\def\textqed{\unskip\nobreak\penalty50
       \hskip2em\hbox{}\nobreak\hfil\qedsymbol
       \parfillskip=0pt \finalhyphendemerits=0}
\def\dispqed{\rlap{\qquad\qedsymbol}}

\opn\dis{dis}
\def\pnt{{\raise0.5mm\hbox{\large\bf.}}}

\opn\Lex{Lex}

\begin{document}
 \title {Gorenstein $T$-spread Veronese algebras}

 \author {Rodica Dinu}

\address{Faculty of Mathematics and Computer Science, University of Bucharest, Str. Academiei 14, 010014 Bucharest, Romania}
\email{rdinu@fmi.unibuc.ro}

%\thanks{The first author was supported by the grant UEFISCDI,  PN-II-ID-PCE- 2011-3-1023.}

 \begin{abstract}
In this paper we characterize the Gorenstein t-spread Veronese algebras.
 \end{abstract}

\thanks{}

 \maketitle

\section*{Introduction}
Let $K$ be a field and let $S=K[x_1, x_2, \dots, x_n]$ be the polynomial ring in $n$ indeterminates over $K$. In the paper \cite{ehq}, Ene, Herzog and Qureshi introduced the concept of $t$-spread monomials. We fix integers $d$ and $t$.
A monomial $x_{i_1}x_{i_2}\cdots x_{i_d}$ with $i_1\leq i_2\leq\dots \leq i_d$ is $t$-spread if $i_j-i_{j-1}\geq t$, for any $2\leq j\leq n$.
Thus any monomial is $0$-spread and a squarefree monomial is $1$-spread.
A monomial ideal in $S$ is called a $t$-spread monomial ideal if it is generated by $t$-spread monomials.
For example, $I=(x_1x_3x_7, x_1x_4x_7, x_1x_5x_8)\subset K[x_1, x_2,\dots, x_8]$ is a $2$-spread monomial ideal.\\
Let $d\geq 1$ be an integer. A monomial ideal in $S$ is called a $t$-spread Veronese ideal of degree $d$ if it is generated by all $t$-spread monomials of degree $d$. We denote it by $I_{n,d,t}$. Note that $I_{n,d,t}\neq 0$ if and only if $n>t(d-1)$.
For example, if $n=5, d=2$ and $t=2$, then
$$
I_{5,2,2}=(x_1x_3, x_1x_4, x_1x_5, x_2x_4, x_2x_5, x_3x_5)\subset K[x_1, x_2, \dots, x_5].
$$
We consider the toric algebra generated by the monomials $v$ with $v\in G(I_{n,d,t})$, here, for a monomial ideal $I$, $G(I)$ denote the minimal system of monomial generators of $I$. This is called a $t$-spread Veronese algebra and we denote it by $K[I_{n,d,t}]$. It generalizes the classical (squarefree)Veronese algebras. By \cite[Corollary 3.4]{ehq}, the $t$-spread Veronese algebra is a Cohen-Macaulay domain.  \\
We fix an integer $d$ and a sequence ${\bf a}=(a_1, \dots, a_n)$ of integers with $1\leq a_1\leq a_2\leq\dots\leq a_n\leq d$ and $d=\sum_{i=1}^n a_i$. The $K$-subalgebra of $S=K[x_1, x_2, \dots, x_n]$ generated by all monomials of the form $t_1^{c_1}t_2^{c_2}\cdots t_n^{c_n}$ with $\sum_{i=1}^n c_i =d$ and $c_i\leq a_i$ for each $1\leq i \leq n$ is called an algebra of Veronese type and it is denoted by $A({\bf a},d)$. If each $a_i=1$, then $A({\bf 1},d)$ is generated by all the squarefree monomials of degree $d$ in $S$.\\
 \indent By \cite[Theorem 2.4]{mainpaper}, in the squarefree case, the algebra of Veronese type $A({\bf 1},d)$ is Gorenstein if and only if
\begin{itemize}
\item [{\em (i)}] $d=n$.
\item [{\em (ii)}] $d=n-1$.
\item [{\em (iii)}] $d<n-1, n=2d$.
\end{itemize}
The aim of this paper is to characterize the toric algebras $K[I_{n,d,t}]$ which have the Gorenstein property. Our approach is rather geometric. We identify the $t$-spread Veronese algebra, $K[I_{n,d,t}]$, with the Ehrhart ring $\mathcal{A}(\mathcal{P})$ associated to a suitable polytope $\mathcal{P}$, and next we employ the conditions which caracterize the Gorenstein property of $\mathcal{A}(\mathcal{P})$. The main result of this paper, Theorem \ref{mainresult}, classifies the $t$-spread Veronese algebras which are Gorenstein. Namely, we show that, for $d,t \geq 2$, $K[I_{n,d,t}]$ is Gorenstein if and only if $n\in \{(d-1)t+1, (d-1)t+2, dt, dt+1, dt+d\}$. We illustrate all our results with suitable examples.

\section{Ehrhart ring of a rational convex polytope}

\indent Let $\mathcal{P}\subset \mathbb{R}^N$ be a convex polytope of dimension $d$ and let $\partial \mathcal{P}$ be the boundary of $\mathcal{P}$. Then $\mathcal{P}$ is called {\it of standard type} if $d=N$ and the origin of $\mathbb{R}^N$ is contained in the interior of $\mathcal{P}$. We call a polytope $\mathcal{P}$ {\it rational} if every vertex of $\mathcal{P}$ has rational coordinates and {\it integral} if every vertex of $\mathcal{P}$ has integral coordinates. The {\it Ehrhart ring} of $\mathcal{P}$ is $\mathcal{A}(\mathcal{P})= \bigoplus_{n\geq 0} \mathcal{A}(\mathcal{P})_n$, where $\mathcal{A}(\mathcal{P})_n$ is the $K$-vector space generated by the monomials $\{x^{a}y^{n}: a\in n\mathcal{P}\cap \mathbb{Z}^n\}$. Here $n\mathcal{P}$ denotes the dilated polytope  $\{(na_1, na_2,\dots, na_d): (a_1, a_2, \dots, a_d)\in \mathcal{P}\}$. It is known that $\mathcal{A}(\mathcal{P})$ is a finitely generated $K$-algebra and a normal domain (\cite[Theorem 9.3.6]{monalg}). The reader can find more about Ehrhart rings of rational convex polytopes in \cite{beck} and \cite{monalg}.

\indent Let $\mathcal{P}\subset \mathbb{R}^d$ be a $d$-dimensional convex polytope of standard type. Then the {\it dual polytope of $\mathcal{P}$} is
\begin{center}
$\mathcal{P}^{*}=\{(\alpha_1,\dots, \alpha_d)\in \mathbb{R}^d: \sum_{i=1}^{d}\alpha_{i}\beta_{i}\leq 1,$ for all $(\beta_{1},\dots, \beta_{d})\in \mathcal{P}\}.$
\end{center}
One can check that $\mathcal{P}^{*}$ is a convex polytope of standard type and $(\mathcal{P}^{*})^{*}=\mathcal{P}$;(see \cite[Exercise 1.14]{brunsg} or \cite[Chapter 2]{ziegler}). It is known the fact that if $(\alpha_1,\dots, \alpha_d)\in \mathbb{R}^d$ and if $H\subset \mathbb{R}^d$ is the hyperplane defined by the equation $\sum_{i=1}^d \alpha_i x_i =1$, then $(\alpha_1, \dots, \alpha_d)$ is a vertex of $\mathcal{P}^{*}$ if and only if $H\cap \mathcal{P}$ is a facet of $\mathcal{P}$, see \cite[Chapter 2]{ziegler}. Therefore, the dual polytope of a rational convex polytope is rational. In order to classify the $t$-spread Veronese algebras which are Gorenstein, we will show that a $t$-spread Veronese algebra coincides with the Ehrhart ring of an integral convex polytope, so we need a criterion for the Ehrhart ring $\mathcal{A}(\mathcal{P})$ to be Gorenstein.\\
Let $\mathcal{P}$ be an integral polytope in $\mathbb{R}^{d}_{+}$ of $\dim \mathcal{P}=d$. We consider the toric ring $K[\mathcal{P}]$ which is generated by all the monomials $x_1^{a_1}\dots x_n^{a_n}s^q$ with $a=(a_1, a_2, \dots, a_n)\in \mathcal{P}\cap \mathbb{Z}^n$ and $q=a_1+a_2+\dots+a_n$. It is known that if $K[\mathcal{P}]$ is a normal ring, then $K[\mathcal{P}]$ is Cohen-Macaulay (\cite[Theorem 6.3.5]{brunsh}).

 \begin{Theorem}(Stanley, Danilov \cite{stanley}, \cite{danilov}) \label{standan} Let $\mathcal{P}\subset \mathbb{R}_{+}^{d}$ be an integral convex polytope and suppose that its toric ring $K[\mathcal{P}]$ is normal, thus $K[\mathcal{P}]=\mathcal{A}(\mathcal{P})$. Then the canonical module $\Omega (K[\mathcal{P}])$ of $K[\mathcal{P}]$ coincides with the ideal of $K[\mathcal{P}]$ which is generated by those monomials $x^{a}s^{q}$ with $a\in q(\mathcal{P}-\partial \mathcal{P})\cap \mathbb{Z}^d$.
\end{Theorem}
By \cite[Proposition A.6.6]{monomialideals}, the Cohen-Macaulay type of a Cohen-Macaulay graded $S$-module $M$ of dimension $d$ coincides with $\beta_{n-d}^{S}(M)$. In particular, a Cohen-Macaulay ring $R=S/I$ is Gorenstein if and only if $\beta_{n-d}^{S}(R)=1$. Let $\mathcal{P}$ be a polytope as in Theorem \ref{standan} and $\delta\geq 1$ be the smallest integer such that $\delta(\mathcal{P}-\partial \mathcal{P})\cap \mathbb{Z}^d \neq \emptyset$. Then, by \cite{noma}, the $\textbf {a}$-invariant is
\begin{center} $\textbf{a}(K[\mathcal{P}])=-\min{(\omega_{K[\mathcal{P}]})\neq 0}=-\delta$.
\end{center}
\begin{Remark}\label{unique}
 By \cite[Corollary A.6.7]{monomialideals}, $K[\mathcal{P}]$ is Gorenstein if and only if $\Omega(K[\mathcal{P}])$ is a principal ideal. In particular, if $K[\mathcal{P}]$ is Gorenstein, then $\delta(\mathcal{P}-\partial \mathcal{P})$ must posses a unique interior vector.
 \end{Remark}
 By \cite[Theorem 1.1]{hibi2}, under the hypothesis that $\mathcal{P}$ is an integral convex polytope, we have the following result:
\begin{Theorem} ({\cite[Theorem 1.1]{hibi2}}) \label{stdtype}
Let $\mathcal{P}$ be a integral convex polytope of dimension $d$ and let $\delta \geq 1$ be the smallest integer for which $\delta(\mathcal{P}-\partial \mathcal{P})\cap \mathbb{Z}^d \neq \emptyset$. Fix $\alpha \in \delta(\mathcal{P}-\partial \mathcal{P})\cap \mathbb{Z}^d \neq \emptyset$ and denote by $\mathcal{Q}$ the rational convex polytope of standard type $\mathcal{Q}=\delta \mathcal{P}-\alpha \subset \mathbb{R}^d$. Then the Ehrhart ring of $\mathcal{P}$ is Gorenstein if and only if the dual polytope $\mathcal{Q}^{*}$ of $\mathcal{Q}$ is integral.
\end{Theorem}

A combinatorial proof of the Theorem \ref{stdtype} can be found in \cite{hibi2} and an algebraic proof of the same theorem can be found in \cite{noma}.

\section{t-spread Veronese algebras}
\indent Let $\mathcal{M}_{n,d,t}$ be the set of $t$-spread monomials of degree $d$ in $n$ variables.\\
\indent To begin with, we study when the $t$-spread Veronese algebra, $K[I_{n,d,t}]$, is a polynomial ring. If $t=1$, then the $t$-spread Veronese algebra coincides with the classical squarefree Veronese algebra and these which are Gorestein are studied in the paper \cite{mainpaper}. Assume $t\geq 2.$ If $n=(d-1)t+1$, then $K[I_{n,d,t}]$ has only one generator, thus it is Gorenstein. Therefore, in what follows, we always consider $t\geq 2$ and $n\geq (d-1)t+2$.
In order to study when the $t$-spread Veronese algebra $K[I_{n,d,t}]$ is a polynomial ring, we need to study sorted sets of monomials, a concept introduced by Sturmfels (\cite{sturmfels}). Let $S_d$ be the $K$-vector space generated by the monomials of degree $d$ in $S$ and let $u, v\in S_d$ be two monomials. We write $uv=x_{i_1}x_{i_2} \dots x_{i_{2d}}$ with $1\leq i_1\leq i_2\leq \dots \leq i_{2d}\leq n$ and define
\begin{center}
$u^{\prime}=x_{i_1}x_{i_3}\dots x_{i_{2d-1}}, v^{\prime}= x_{i_2}x_{i_4}\dots x_{i_{2d}}.$
\end{center}
The pair $(u^{\prime},v^{\prime})$ is called the {\it sorting} of $(u,v)$ and the map
\begin{center}
$\sort: S_d\times S_d \rightarrow S_d\times S_d, (u,v)\mapsto (u^{\prime}, v^{\prime})$
\end{center}
is called the {\it sorting operator}. A pair $(u,v)$ is {\it sorted} if $\sort(u,v)=(u,v)$. For example, $(x_1^2x_2x_3, x_1x_2x_3^2)$ is a sorted pair. Notice that if $(u,v)$ is sorted, then $u>_{lex}v$ and $\sort(u,v)=\sort(v,u)$. If $u_1=x_{i_1}\dots x_{i_d}, u_2=x_{j_1}\dots x_{j_d}, \dots, u_{r}=x_{l_1}\dots x_{l_d}$, then the $r$-tuple $(u_1, \dots, u_r)$ is sorted if and only if
\begin{center}
$i_1\leq j_1\leq  \dots\leq  l_1\leq i_2\leq j_2\leq \dots \leq l_2\leq \dots \leq i_d\leq j_d\leq \dots \leq l_d,$
\end{center}
which is equivalent to $(u_i, u_j)$ sorted, for all $i>j$.

\begin{Proposition}\label{sorted}
Let $u_1,\dots,u_q$ be the generators of $K[I_{n,d,t}]$. If $n=(d-1)t+2$, then any $r-$tuple $(u_1,\dots,u_r)$ with $u_1\geq_{lex} u_2\geq_{lex}\cdots \geq_{lex} u_r$ is sorted.
\end{Proposition}
\begin{proof}
It suffices to show that any pair $(u_i, u_j)$ with $u_i >_{lex} u_j$ is sorted. Let $u_i=x_{i_1}x_{i_2}\cdots x_{i_d}$ with $i_k-i_{k-1}\geq t$, for any $k\in \{2,\dots,d\}$ and $v_j=x_{j_1}x_{j_2}\cdots x_{j_d}$ with $j_k-j_{k-1}\geq t$, for any $k\in \{2, \ldots, d\}$. Since $n= (d-1)t+2$, the smallest monomial is $x_2x_{t+2}\cdots x_{(d-1)t+2}$ and the largest monomial is $x_1x_{t+1}\cdots x_{(d-1)t+1}$, with respect to lexicographic order. Then $1+jt\leq i_{j+1}\leq 2+jt$, for any $j\in \{0,1,\ldots, d-1\}$. Since $u_i>_{lex} u_j$, we have
\begin{center}
$u_i= x_1\dots x_{1+(k-1)t}x_{2+kt}\dots x_{2+(d-1)t}$ and
\end{center}
\begin{center}
$u_j= x_1\dots x_{1+(l-1)t}x_{2+lt}\dots x_{2+(d-1)t}$, for some $k>l.$
\end{center}
Then one easily sees that $(u_i, u_j)$ is always sorted.
\end{proof}

\begin{Corollary}\label{polyring}
Let $n\geq (d-1)t+2$. The $t$-spread Veronese algebra $K[I_{n,d,t}]$ is a polynomial ring if and only if $n=(d-1)t+2$. In particular, $K[I_{n,d,t}]$ is Gorestein if $n=(d-1)t+2$.
\end{Corollary}

\begin{proof}
Let $u_1, \dots, u_q$ be the generators of $K[I_{n,d,t}]$. We want to show that these elements are algebraically independent over the field $K$. Let $f=\sum_{\alpha} a_{\alpha}y_1^{\alpha_1}y_2^{\alpha_2}\cdots y_q^{\alpha_q}$ be a polynomial such that $f(u_1, \dots, u_q)=0$. By Proposition \ref{sorted}, any $r$-tuple $(u_1,\dots, u_r)$ of generators with $u_1\geq_{lex}\dots \geq_{lex}u_r$ is sorted, which implies that the monomials $u_1^{\alpha_1}\dots u_q^{\alpha_q}$ are all pairwise distinct. Then the coefficients $a_{\alpha}$ are all zero, which implies that $u_1,\dots,u_q$ are algebraically independent over $K$.\\
For the converse part, assume that there exists $n\geq (d-1)t+3$ such that $K[I_{n,d,t}]$ is a polynomial ring. Then it is clear that if $u=x_1x_{t+1}\cdots x_{(d-2)t+1}x_n$ and $v=x_2x_{t+2}\cdots x_{(d-2)t+2}x_{n-1}$, then $(u,v)$ is unsorted and the pair $(u^{\prime}, v^{\prime})$, where $u^{\prime}=x_1x_{t+1}\cdots x_{(d-2)t+1}x_{n-1}$ and $v^{\prime}=x_2x_{t+2}\cdots x_{(d-2)t+2}x_n$ is the sorting pair of $(u,v)$ and the equality $uv-u^{\prime}v^{\prime}$ gives a non-zero polynomial in the defining ideal of $K[I_{n,d,t}]$, contradicting the fact that $K[I_{n,d,t}]$ is a polynomial ring.
\end{proof}

Moreover, we can make a stronger reduction. Let $n<dt$. Then the smallest $t$-spread monomial of degree $d$ is $x_{n-(d-1)t}x_{n-(d-2)t}\dots x_n$. As $n-(d-1)t<t$, the generators of $I_{n,d,t}$ can be viewed in a polynomial ring in the variables $\{x_1, \dots, x_n\}\setminus \cup_{l=1}^{d-1}\{x_{n-dt+lt+1}, \dots, x_{lt}\}$. Thus $K[I_{n,d,t}]\subset S^{\prime}$, where $$S^{\prime}=K[\{x_1, \dots, x_n\}\setminus \cup_{l=1}^{d-1}\{x_{n-dt+lt+1}, \dots, x_{lt}\}],$$ which is a polynomial ring in $n^{\prime}=n-(d-1)(dt-n)=d(n-(d-1)t)$ variables. Note that, in $S^{\prime}$, $I_{n,d,t}$ is a $t^{\prime}$-spread ideal, where $t^{\prime}=n-(d-1)t$. Thus, $n^{\prime}=dt^{\prime}$. This discussion shows that, in what follows, we may consider $n\geq dt$.

\begin{Theorem}\label{dimalg}
\begin{itemize}
\item [{\em (i)}] If $n\geq dt+1$, then $\dim K[I_{n,d,t}]= n$.
\item [{\em (ii)}] If $n=dt$, then $\dim K[I_{n,d,t}]=n-d+1$.
\end{itemize}
\end{Theorem}

\begin{proof}
(i). We denote by $y_i$ the $d$-th power of the variable $x_i$, for any $1\leq i\leq n$. Let $A=K[I_{n,d,t}]$. We prove that $y_1, y_2, \dots, y_n$ belong to the quotient field of $A$, denoted by $Q(A)$. We first show by induction on $0\leq k\leq d-1$ that $y_{kt+j}\in Q(A)$, for any $1\leq j\leq t$.\\
\indent We check for $k=0$: it is clear that
\begin{center}
$y_1=x_1^d=\frac{\prod_{j=1}^{d}x_{1}x_{t+1}\dots \widehat{x_{jt+1}}\dots x_{td+1}}{(x_{t+1}\dots x_{dt+1})^{d-1}}\in Q(A).$
\end{center}
\indent Here, by $\widehat{x_{jt+1}}$, we mean that the variable $x_{jt+1}$ is missing.\\
 Since $y_1y_{t+j}\dots y_{(d-1)t+j}\in Q(A)$, for $1\leq j\leq t$, we get $y_{t+j}\dots y_{(d-1)t+j}\in Q(A)$, for $1\leq j\leq t$. But also $y_jy_{t+j}\dots y_{(d-1)t+j}\in Q(A)$, so we obtain $y_j\in Q(A)$, for any $1\leq j\leq t$. Therefore, it follows $y_1, \dots, y_t$ belong to $Q(A)$.\\
\indent Assume that $y_1, y_2, \dots, y_t, \dots, y_{kt+1}, \dots, y_{(k+1)t}\in Q(A)$. We want to prove that $y_{(k+1)t+1}, \dots, y_{(k+2)t}$ are also in $Q(A)$. Firstly, let us check if $y_{(k+1)t+1}$ belongs to $Q(A)$. Notice that, since $y_{t+1}y_{2t+1}\dots y_{kt+1}y_{(k+1)t+1}\dots y_{dt+1}$ and $y_{t+1}, y_{2t+1}, \dots, y_{kt+1}\in Q(A)$ by our assumption, it follows that $y_{(k+1)t+1}\dots y_{dt+1}\in Q(A)$. \\
Also, since $y_1y_{2t+1}\dots y_{kt+1}y_{(k+2)t+1}\dots y_{dt+1}\in Q(A)$, using our assumption, we get $y_{(k+2)t+1}\dots y_{dt+1}\in Q(A)$. But since $y_{(k+1)t+1}\dots y_{dt+1}$ is in $Q(A)$, it follows that $y_{(k+1)t+1}\in Q(A)$.\\
 \indent Now we check that $y_{(k+1)t+s}\in Q(A)$, for any $2\leq s\leq t$. Using the monomials $y_s\dots y_{kt+s}y_{(k+1)t+s}\dots y_{(d-1)t+s}\in Q(A)$ and $y_s, \dots, y_{kt+s}\in Q(A)$ by our assumption, we get $$y_{(k+1)t+s}\dots y_{(d-1)t+s}\in Q(A).$$
 Moreover, $y_1\dots y_{kt+1}y_{(k+1)t+1}y_{(k+2)t+s}\dots y_{(d-1)t+s}$ is in $Q(A)$, so by our assumption and by the fact that $y_{(k+1)t+1}\in Q(A)$, we obtain $$y_{(k+2)t+s}y_{(k+3)t+s}\dots y_{(d-1)t+s}\in Q(A).$$
 Therefore, using $y_{(k+1)t+s}\dots y_{(d-1)t+s}$ and $y_{(k+2)t+s}\dots y_{(d-1)t+s}$ in $Q(A)$, we get $$y_{(k+1)t+s}\in Q(A),$$ for any $2\leq s\leq t$. So far, we have seen that $y_{kt+j}\in Q(A)$, also for any $0\leq k\leq d-1$ and $1\leq j\leq t$. Let now $dt+1\leq m\leq n$. Then $y_1y_{t+1}\dots y_{(d-1)t+1}y_m\in Q(A)$. Since $y_1, y_{t+1}, \dots, y_{(d-1)t+1}\in Q(A)$, it follows that $y_m\in Q(A)$ as well. Therefore, $Q(A)\supset \{x_1^d, \dots, x_n^d\}$. It follows that $\dim A= \trdeg Q(A)\geq n$, since $x_1^d, \dots, x_n^d$ are obviously algebraic independent over $K$. But since $A$ is a subalgebra of $K[x_1, \dots, x_n]$, by \cite[Proposition 3.1]{binomialideals}, $\dim A \leq n$. Therefore, $\dim A=n$.\\
(ii). It follows from \cite[Corollary 3.2]{bahareh}.\\
\end{proof}
\begin{Remark}
The result from Theorem \ref{dimalg} (i) follows from \cite[Corollary 3.2]{bahareh}, but we prefered to give a completely different proof here.
\end{Remark}

\indent Let $\mathcal{P}\subset \mathbb{R}^n$ denote the rational convex polytope
\begin{center}
$\mathcal{P}=\{(a_1,\dots, a_n)\in \mathbb{R}^n: \sum_{i=1}^{n}a_i=d, a_i\geq 0,$ for $1\leq i\leq n,$ and $a_{i}+\dots + a_{i+t-1}\leq 1$, for $1\leq i\leq n-t+1\}.$
\end{center}

\begin{comment}
\begin{Definition}Let $n\geq 1$ an integer.
A polytope $\mathcal{P}$ has the integer decomposition property if for all $N\geq 1$ and for all $\alpha \in N\mathcal{P}\in \mathbb{Z}^n$, there exist $\alpha_1, \dots, \alpha_N \in \mathcal{P}\cap \mathbb{Z}^n$, not all distict, with $\alpha=\alpha_1+\dots +\alpha_N$.
\end{Definition}
By \cite[Proposition 3.1]{ehq}, the set $G(K[I_{n,d,t}])$ is sortable, thus the polytope $\mathcal{P}$ has the integer decomposition property. Thus, for any $q\geq 1$, any point belonging to the dilated polytope $q\mathcal{P}\cap \mathbb{Z}^n$ is the sum of $q$ elements in $\mathcal{P}\cap \mathbb{Z}^n$. Hence, the convex polytope $\mathcal{P}$ is an integral convex polytope of dimension $\dim K[I_{n,d,t}]-1$.
\end{comment}
Clearly, $K[I_{n,d,t}]=K[\mathcal{P}]$, since $K[I_{n,d,t}]$ is generated by the monomials of $G(I_{n,d,t})$, that is, by monomials $x_1^{a_1}\dots x_n^{a_n}$ with $\sum_{i=1}^{n}a_i=d$,$a_i\geq 0$, for $1\leq i\leq n$ and $a_i+\dots +a_{i+t-1}\leq 1$, for $1\leq i\leq n-t+1$.\\
\indent Since $K[I_{n,d,t}]$ is a normal ring, by \cite[Lemma 4.22]{binomialideals}, we get the following
\begin{Theorem}
The $t$-spread Veronese algebra $K[I_{n,d,t}]$ is the Ehrhart ring $\mathcal{A}(\mathcal{P})$.
\end{Theorem}

\begin{comment}
\begin{proof}
Since, for any $q\geq 1$, any point belonging to the dilated polytope $q\mathcal{P}\cap \mathbb{Z}^n$ is the sum of $q$ elements in $\mathcal{P}\cap \mathbb{Z}^n$, the Ehrhart ring $\mathcal{A}(\mathcal{P})$ of $\mathcal{P}$ is generated by $(\mathcal{A}(\mathcal{P}))_1$ as an algebra over $K=(\mathcal{A}(\mathcal{P}))_0$, thus $\mathcal{A}(\mathcal{P})$ is generated by the monomials $x_1^{a_1}\cdot \dots \cdot x_n^{a_n}$ with $(a_1, \dots, a_n)\in \mathcal{P}\cap \mathbb{Z}^n$ as an algebra over $K$. The toric algebra $K[I_{n,d,t}]$ is generated by the monomials $v$ with $v\in G(I_{n,d,t})$, thus is generated by the monomials $x_1^{a_1}\cdot \dots \cdot x_n^{a_n}$ such that
 \begin{center}$\sum_{i=1}^{n}a_i=d, a_i\geq 0, 1\leq i\leq n, a_{i}+\dots + a_{i+t-1}\leq 1, 1\leq i\leq n-t+1,$
 \end{center}
  hence $(a_1, \dots, a_n)\in \mathcal{P}\cap \mathbb{Z}^n$. Therefore, $\mathcal{A}(\mathcal{P})\cong K[I_{n,d,t}]$.
\end{proof}

\begin{Remark}
From \cite[Corollary 3.4]{ehq}, $K[I_{n,d,t}]$ is a Cohen-Macaulay normal domain. It is known the fact that the Ehrhart ring $\mathcal{A}(\mathcal{P})$ associated with a rational convex polytope is a Cohen-Macaulay normal domain of Krull-dimension $\dim \mathcal{P}+1$. Thus, we obtain in this way, the same result as in \cite[Corollary 3.4]{ehq}.
\end{Remark}

\end{comment}

\section{Gorenstein t-spread Veronese algebras}
\indent In this section we classify the Gorenstein $t$-spread Veronese algebras. We split the classification in several theorems.
\begin{Theorem}\label{n1}
If $n=dt+k$, $2\leq k\leq d-1$, then in $(t+d)P$ there exist $d$ interior lattice points. Therefore, $K[I_{n,d,t}]$ is not Gorenstein.
\end{Theorem}

\begin{proof} By Theorem \ref{dimalg}, $\dim K[I_{n,d,t}]=n$, thus $\dim (\mathcal{P})=n-1$.
 Let $H$ be the hyperplane in $\mathbb{R}^n$ defined by the equation $a_1+\dots+a_n=d$ and let $\phi : \mathbb{R}^{n-1}\rightarrow H$ denote the affine map defined by $$
 \phi (a_1, \dots, a_{n-1})= (a_1, \dots, a_{n-1}, d-(a_1+\dots+ a_{n-1})),
 $$
 for $(a_1, \dots, a_{n-1})\in \mathbb{R}^{n-1}$. Then $\phi$ is an affine isomorphism and $\phi (\mathbb{Z}^{n-1})=H\cap \mathbb{Z}^{n}$. Therefore, $\phi ^{-1}(\mathcal{P})$ is an integral convex polytope in $\mathbb{R}^{n-1}$ of $\dim \phi^{-1}(\mathcal{P})=\dim \mathcal{P}=n-1$. The Ehrhart ring $\mathcal{A}(\phi^{-1}(\mathcal{P}))$ is isomorphic with $\mathcal{A}(\mathcal{P})$ as graded algebras over $K$. Thus, we want to see if $\mathcal{A}(\phi^{-1}(\mathcal{P}))$ is Gorenstein, and, by abuse of notation, we write $\mathcal{P}$ instead of $\phi^{-1}(\mathcal{P})$. Thus,
 \begin{center}
 $\mathcal{P}= \{(a_1, \dots, a_{n-1})\in \mathbb{R}^{n-1}: a_i\geq 0$,for $1\leq i\leq n-1,$ and $a_i+a_{i+1}+\dots + a_{i+t-1}\leq 1$, for $1\leq i\leq n-t, a_1+a_2+\dots+ a_{n-t}\geq d-1\}$.
 \end{center}
 \begin{comment}
 Let $\delta \geq 1$ the smallest integer such that $\delta (\mathcal{P}-\partial \mathcal{P})\cap \mathbb{Z}^{n-1}\neq \emptyset$. If $\gamma>\delta$, then
$$
\delta(\mathcal{P}-\partial \mathcal{P})\cap \mathbb{Z}^{n-1}+ (\gamma-\delta)(\mathcal {P}\cap \mathbb{Z}^{n-1}) \subset \gamma(\mathcal{P}-\partial \mathcal{P}),
$$
thus, since $\mathcal{P}$ is an integral convex polytope, $|\gamma(\mathcal{P}-\partial \mathcal{P})\cap \mathbb{Z}^{n-1}|>1$. If $\mathcal{A}(\mathcal{P})$ is Gorenstein, then $$|\delta(\mathcal{P}-\partial \mathcal{P})\cap \mathbb{Z}^{n-1}|=1.$$
\end{comment}

 In our hypothesis on $n$, we show that there are no interior lattice points at lower levels than $t+d$. It is enough to see that there are no interior lattice points at level $t+d-1$. Let $(x_1,\dots,x_{n-1})\in (t+d-1)(\mathcal{P}-\partial \mathcal{P})$. We have
 \begin{center}
 $(t+d-1)(\mathcal{P}-\partial \mathcal{P})= \{(a_1, \dots, a_{n-1})\in \mathbb{R}^{n-1}: a_i> 0, 1\leq i\leq n-1, a_i+a_{i+1}+\dots + a_{i+t-1}< (t+d-1), 1\leq i\leq n-t, a_1+a_2+\dots+ a_{n-t}> (d-1)(t+d-1)\}$.
 \end{center}
  Since $x_1+\dots+x_{(d-1)t+k}\geq (d-1)(t+d-1)$, we have $x_1+\dots+ x_{(d-1)t+k-1}\geq (d-1)(t+d-1)-x_{(d-1)t+k}$, thus
$$ (d-1)(t+d-2)+\sum_{i=1}^{k-1}x_{(d-1)t+i}\geq \sum_{i=1}^{(d-1)t+k-1}x_i\geq (d-1)(t+d-1)-x_{(d-1)t+k}. $$ It follows that $\sum_{i=1}^{k-1}x_{(d-1)t+i}\geq d-1-x_{(d-1)t+k}$ and, since $x_{(d-1)t+k}<1$, we obtain
$$ k-1>\sum_{i=1}^{k-1}x_{(d-1)t+i}\geq d-1,$$ which implies that $k\geq d+1$, a contradiction.
 Thus, there are no interior lattice points in the dilated polytope at lower levels than $t+d$.\\
 \indent Let $\delta \geq 1$ be the smallest integer such that $\delta(\mathcal{P}-\partial \mathcal{P})\neq \emptyset$. We show that $\delta=t+d$.
Indeed, in the dilated polytope $(t+d)(\mathcal{P}-\partial \mathcal{P})$ there are $d$ interior lattice points of the form $\alpha_{r}=(x_1^{(r)}, \dots, x_{n-1}^{(r)})$, $1\leq r\leq d$, where

$x_{j}^{(r)} = \begin{cases} d,  \mbox{if } j=it+1, \\ 1,  \mbox{if } j=it+l\mbox{ with } 0\leq i\leq d-1, 1<l\leq t \\\mbox{ or } j=dt+l \mbox{ with } 1\leq l\leq dt+k-2,\\ r, \mbox{if} j=dt+k-1.\end{cases}$

 It is clear that, for any $1\leq j\leq n-1$, $x_j^{(r)}>0.$ For any $1\leq i\leq (d-1)t+1$, $x_i^{(r)}+x_{i+1}^{(r)}+\dots + x_{i+t-1}^{(r)}=d+t-1<(t+d)(d-1)$. If $k<t$, then $x_{(d-1)t+t-k}^{(r)}+\dots+ x_{dt+k-1}^{(r)}=(t-k)+(k-1)+r=t-1+r<t+d$ and, if $k\geq t$, then $x_{dt+k-t}^{(r)}+\dots+ x_{dt+k-1}^{(r)}=t-1+r<t+d$. Also, $x_1^{(r)}+\dots+ x_{n-t}^{r}=(d+t-1)(d-1)+d+k-1= (d+t)(d-1)+k>(d-1)(t+d)$, since $k\geq 2$.
  Therefore, these are interior lattice points in $(t+d)\mathcal{P}$. Thus, in this case, the $t$-spread Veronese algebra $K[I_{n,d,t}]$ is not Gorenstein, by Remark \ref{unique}.
\end{proof}

\begin{Example}
Let $n=8$, $d=3$ and $t=2$. The smallest level where there are interior lattice points in the dilated polytope is $\delta=5$. In $5(\mathcal{P}-\partial \mathcal{P})$ there are $3$ interior lattice points: $(3,1,3,1,3,1,1), (3,1,3,1,3,1,2)$ and $(3,1,3,1,3,1,3)$.\\
Thus, the $2$-spread Veronese algebra $K[I_{8,3,2}]$ is not Gorenstein.
\end{Example}

\begin{Theorem}\label{n2}
If $n\geq (t+1)d+1$, then $K[I_{n,d,t}]$ is not Gorenstein.
\end{Theorem}
\begin{proof} Let $n=kd+q$ with $k\geq t+1$ and $q\geq 1$.
By Theorem \ref{dimalg}, $\dim K[I_{n,d,t}]=n$, thus $\dim (\mathcal{P})=n-1$. Using similar arguments as in Theorem \ref{n1},
\begin{center}
 $\mathcal{P}= \{(a_1, \dots, a_{n-1})\in \mathbb{R}^{n-1}: a_i\geq 0, 1\leq i\leq n-1, a_i+a_{i+1}+\dots + a_{i+t-1}\leq 1, 1\leq i\leq n-t, a_1+a_2+\dots+ a_{n-t}\geq d-1\}$.
 \end{center}
\indent We show that the smallest integer $\delta\geq 1$ such that $\delta(\mathcal{P}-\partial \mathcal{P})$ contains lattice points is $t+1$. Assume that there are interior lattice points at lower levels than $t+1$. It is enough to show that there are no interior lattice points at level $t$. In this case, for each lattice point $(a_1, \dots, a_{n-1})\in t(\mathcal{P}- \partial \mathcal{P})$, we have $a_{i}+a_{i+1}+\dots+a_{i+t-1}<t$, for any $1\leq i\leq n-t$. Since each $a_i \geq 1$, for any $1\leq i \leq n-1$, we have
\begin{center}
$a_i+a_{i+1}+\dots+a_{i+t-1}\geq t,$
\end{center}
 which is a contradiction. \\
 \indent We show that $(t+1)(\mathcal{P}-\partial \mathcal{P})$ contains only one lattice point which has all the coordinates equal to $1$. The interior of the $(t+1)$-dilated polytope is
 \begin{center}
 $(t+1)(\mathcal{P}-\partial \mathcal{P})= \{(a_1, \dots, a_{n-1})\in \mathbb{R}^{n-1}: a_i>0, 1\leq i\leq n-1, a_i+a_{i+1}+\dots + a_{i+t-1}<t+1, 1\leq 1\leq n-t, a_1+a_2+\dots+ a_{n-t}> (d-1)(t+1)\}$.
 \end{center}
 We know that, for each lattice point $(a_1, \dots, a_{n-1})\in (t+1)(\mathcal{P}-\partial \mathcal{P})$, we have $a_i+\dots +a_{i+t-1}\leq t$, for any $1\leq i\leq n-t$ and, since $a_j\geq 1$, for any $i\leq j\leq i+t-1$, we obtain $a_i+\dots+a_{i+t-1}\geq t$, thus we have equality which implies that $a_j=1$, for any $1\leq j\leq n-1$. Hence, $(1,1,\dots, 1)\in \mathbb{Z}^{n-1}$ is the unique interior lattice point in the dilated polytope $(t+1)\mathcal{P}$. Let us consider $\mathcal{Q}=(t+1)\mathcal{P}-(1,1,\dots, 1)$. We will show that $K[\mathcal{P}]$ is not Gorenstein by using Theorem \ref{stdtype}. In fact, we show that the dual $\mathcal{Q}^{*}$ of
 \begin{center}
$\mathcal{Q}=\{(y_1, \dots, y_{n-1})\in \mathbb{R}^{n-1}: y_i\geq 0, 1\leq i\leq n-1, y_i+y_{i+1}+\dots + y_{i+t-1}\leq 0, 1\leq i\leq n-t, y_1+ \dots + y_{n-t}\geq (d-1)(t+1)-(n-t)\}.$
\end{center}
is not an integral polytope.
The vertices of $\mathcal{Q}^{*}$ are of the form $(a_1, a_2,\dots, a_{n-1})\in \mathbb{R}^{n-1}$ such that the hyperplane $H$ of equation $\sum_{i=1}^{n-1}a_iy_i=1$ has the property that $H\cap \mathcal{Q}$ is a facet of $\mathcal{Q}$. In other words, $H$ is a supporting hyperplane of $\mathcal{Q}$. As the hyperplane $\sum_{i=1}^{n-t}y_i=(d-1)(t+1)-(n-t)$, that is, $$\sum_{i=1}^{n-t}\frac{1}{(d-1)(t+1)-(n-t)}y_i=1$$ does not have integral coefficients, it follows that $\mathcal{Q}$ is not an integral polytope. Thus, by Theorem \ref{stdtype}, we conclude that $K[I_{n,d,t}]$ is not Gorenstein.
\end{proof}

\begin{Example}
Let $n=10$, $d=3$ and $t=2$. Then
$$
\mathcal{P}=\{(a_1, \dots, a_9)\in \mathbb{R}^{9}: a_i\geq 0, 1\leq i \leq 9\}, a_i+a_{i+1}\leq 1, 1\leq i\leq 8, a_1+\dots+a_8\geq 2 \}.
$$
For $\delta=3$, in $3\mathcal{P}$, there exists a unique interior lattice point, namely $(1,1,1,1,1,1,1,1)$. Let us compute $\mathcal{Q}=3\mathcal{P}-(1,1,1,1,1,1,1,1,1)$. We have
$$
\mathcal{Q}=\{(y_1, \dots, y_9)\in \mathbb{R}^{9}: y_i\geq -1, 1\leq i\leq 9, y_i+y_{i+1}\leq -1, 1\leq i\leq 8, y_1+\dots+y_8\geq -6 \},
$$
thus, the dual polytope $\mathcal{Q}^{*}$ is not integral. Therefore, the $2$-spread Veronese algebra $K[I_{10, 3,2}]$ is not Gorenstein.
\end{Example}

\begin{Theorem}\label{dt}
If $n=dt$, then $K[I_{n,d,t}]$ is Gorenstein.
\end{Theorem}
\begin{proof} In our hypothesis, by Theorem \ref{dimalg}, $\dim K[I_{n,d,t}]=n-d+1=d(t-1)+1$, thus $\dim \mathcal{P}=d(t-1)$. Using similar arguments as in Theorem \ref{n1},
\begin{center}
$\mathcal{P}= \{(a_1, a_2, \dots, a_{n-1})\in \mathbb{R}^{n-1}: a_i\geq 0, a_i+a_{i+1}+\dots+a_{i+t-1}\leq 1, 1\leq i\leq n-t, a_1+a_2+\dots+ a_{n-t}\geq d-1\}.$
\end{center}
\indent We show that the smallest integer $\delta\geq 1$ such that $\delta(\mathcal{P}-\partial \mathcal{P})$ contains lattice points is $t+d-1$. Assume that there are interior lattice points at lower levels than $t+d-1$. It is enough to show that there are no interior lattice points at level $t+d-2$. In this case, for each lattice point $(a_1, \dots, a_{n-1})\in (t+d-2)(\mathcal{P}- \partial \mathcal{P})$, we have $a_1+a_2+\dots+a_{n-t}> (d-1)(t+d-2)$. Since $a_{i}+a_{i+1}+\dots+a_{i+t-1}<t+d-2$, for any $1\leq i\leq n-t$, we obtain $a_1+a_2+\dots+a_{n-t}< (d-1)(t+d-2)$, which leads to contradiction. Thus $\delta \geq t+d-1$.\\
 \indent We show that $(t+d-1)\mathcal{P}$ contains a unique interior lattice point. For each lattice point $(a_1, a_2, \dots, a_{n-1})\in (t+d-1)\mathcal{P}$, we have $a_1+a_2+\dots +a_{n-t}\geq (d-1)(t+d-1)$. Since $a_i+a_{i+1}+\dots+a_{i+t-1}\leq t+d-1$, for any $1\leq i\leq n-t$, we obtain $a_1+a_2+\dots+ a_{n-t}\leq (d-1)(t+d-1)$, thus $a_1+a_2+\dots+ a_{n-t}=d+t-1$. Hence we obtain $a_{kt+1}+a_{kt+2}+\dots+ a_{kt+t}=t+d-1$, for any $0\leq k\leq d-2$. So, $a_{i}+a_{i+1}+\dots+a_{i+t-1}=t+d-1$, for any $1\leq i\leq t(d-1)$, with $i\equiv 1(\mod t)$. Since $a_{kt+1}+a_{kt+2}+\dots+a_{kt+t}=t+d-1$, for any $0\leq k\leq d-2$, we have
\begin{center}
$ a_{kt+2}= t+d-1-\sum_{ j\neq kt+2} a_j$, for any $0\leq k\leq d-2$.
\end{center}
Since $a_{kt+2}+a_{kt+3}+\dots+ a_{kt+t+1}\leq t+d-1$, for any $0\leq k\leq d-2$, we obtain
$$ (t+d-1) -\sum_{ j\neq kt+2} a_j+ a_{kt+3}+\dots +a_{(k+1)t+1}\leq (t+d-1).$$ Hence, $a_{(k+1)t+1}-a_{kt+1}\leq 0$, for any $0\leq k\leq d-2$, thus $a_{kt+t+1}\leq a_{kt+1}$, for any $0\leq k\leq d-2$. Therefore, for each lattice point $(x_1, x_2, \dots, x_{d(t-1)})\in (t+d-1)(\mathcal{P}-\partial \mathcal{P})$, we obtain $0<x_{(d-1)+1}<\dots< x_{t+1}< x_1<t+d-1$ and, since there are $d$ consecutive terms in this chain, we have $x_1\geq d$. If $x_1 >d$, and since each $x_i>1$, for any $1\leq i\leq n-1$, then $x_1+x_2+\dots+x_t >d+t-1$, which is a contradiction. Thus, $x_1=d$. Since $x_1+x_2+\dots + x_{t}=d+t-1$, $x_1=d$ and $x_i \geq 1$, for any $1\leq i\leq t$, we obtain $x_2=\dots=x_t=1$. Now, since $0<x_{(d-1)t+1}<\dots < x_{t+1}<x_1=d$ and $x_i+x_{i+1}+\dots+x_{i+t-1}<t+d-1$, we obtain $x_{kt+1}=d-k$, for any $0\leq k\leq d-2$ and $x_{kt+j}=1$, for any $0\leq k\leq d-2$ and $0\leq j\leq t, j\neq 1, j\neq 2$.
Therefore, $\alpha=(x_1, x_2, \dots, x_{d(t-1)})$, where
\begin{center}
$x_{j} = \begin{cases} d-k, & \mbox{if } j=kt+1, \mbox{with } 0\leq k\leq d-2, \\ 1,  &\mbox{if } j=kt+l,\mbox{with } 0\leq k\leq d-2, 0\leq l\leq t-1, l\neq 1, l\neq 2\end{cases}$
\end{center}
is the unique interior lattice point in $(t+d-1)\mathcal{P}$. But, for any $0\leq k\leq d-2$ and $kt+1\leq j\leq kt+t,$
 \begin{equation*}
\begin{split}
x_{kt+2} & = t+d-1-\sum_{j\neq kt+2}x_j\\
 & = t+d-1-(d-k+t-2)=k+1.
\end{split}
\end{equation*}
So, the unique interior lattice point $\alpha$ in $(t+d-1)\mathcal{P}$ is $(x_1, \dots, x_{n-1})$, where

 \[
x_{j}= \left\{
\begin{array}{lll}
      d-k, & j=kt+1, 0\leq k\leq d-2,\\
      k+1, & j=kt+2, 0\leq k\leq d-2, \\
      1, & j=kt+l, 0\leq k\leq d-2, 0\leq l\leq t-1, l\neq 1, l\neq 2.
\end{array}
\right.
\]
Using Theorem \ref{stdtype}, we show that $K[I_{n,d,t}]$ is Gorenstein. Let us compute $\mathcal{Q}=(t+d-1)\mathcal{P}-\alpha$. We have
 \begin{center}
$\mathcal{Q}=\{(y_1, \dots, y_{n-1})\in \mathbb{R}^{n-1}:  y_i= a_i-x_i, 1\leq i\leq n-1$, where $(a_1, a_2, \dots, a_{n-1})\in (t+d-1)\mathcal{P}\}.$
\end{center}
Thus, we obtain
\begin{center}
$\mathcal{Q}=\{(y_1, \dots, y_{n-1})\in \mathbb{R}^{n-1}:  y_i\geq -x_i, 1\leq i\leq n-1, y_i+y_{i+1}+\dots+y_{i+t-1}\leq t+d-1-(x_i+x_{i+1}+\dots+x_{i+t-1}), 1\leq i\leq n-t, y_1+\dots+y_{n-t}=0, y_{kt+1}+y_{kt+2}+\dots+ y_{kt+t}=0, 0\leq k\leq d-2\}.$
\end{center}
 For $1\leq i\leq n-t$, we have $y_i+y_{i+1}+\dots+y_{i+t-1}\leq t+d-1-(x_i+x_{i+1}+\dots+ x_{i+t-1})$, suppose $i=kt+r$, where $1\leq r\leq t$. Then $y_{kt+r}+y_{kt+r+1}+\dots + y_{(k+1)t+r-1}\leq t+d-1-(x_{kt+r}+\dots+ x_{(k+1)t+r-1})$, for any $0\leq k\leq d-2, 1\leq r\leq t$. If $r=1$, we already have $y_{kt+1}+\dots+ y_{kt+t}=0$. If $r=2$, then
$$
y_{kt+2}+y_{kt+3}+\dots+y_{(k+1)t+1}\leq (t+d-1)- (k+1+(t-2)+d-(k+1))=1.
$$
But $y_{kt+2}=-y_{kt+1}-\dots -y_{kt+t}$, thus
$$
y_{(k+1)t+1}-y_{kt+1}\leq 1.
$$
If $r\geq 3$, then
$$
y_{kt+r}+y_{kt+r+1}+\dots+ y_{(k+1)t+r-1}\leq (t+d-1)-(t-2+d-(k+1)+k+2)=0.
$$
But $y_{kt+r}=-y_{kt+1}-\dots y_{kt+r-1}-y_{kt+r+1}-\dots - y_{kt+t}$, thus
$$
y_{(k+1)t+1}+\dots +y_{(k+1)t+r-1}-y_{kt+1}-y_{kt+2}-\dots -y_{kt+r-1}\leq 0.
$$
Therefore,
\begin{center}
$\mathcal{Q}=\{(y_1, \dots, y_{n-1})\in \mathbb{R}^{n-1}: y_{(k+1)t+1}+\dots+y_{(k+1)t+r-1}-y_{kt+1}-\dots - y_{kt+r-1}\leq 0, 3\leq r\leq t, y_{(k+1)t+1}-y_{kt+1}\leq 1, y_{kt+2}=-y_{kt+1}-y_{kt+3}-\dots -y_{kt+t}, 0\leq k\leq d-2\}.$
\end{center}
Thus, since the supporting hyperplanes of the polytope $\mathcal{Q}$ have integral coefficients, we conclude that $\mathcal{Q}$ is an integral polytope. Hence, by Theorem \ref{stdtype}, $K[I_{n,d,t}]$ is Gorenstein.
\end{proof}

\begin{Example}
Let $n=10$, $d=5$ and $t=2$. In this case, $\delta=6$ and in the dilated polytope
 $6\mathcal{P}$ there is a unique interior lattice point, namely $(5,1,4,2,3,3,2,4,1)$. The dual polytope of $\mathcal{Q}=6\mathcal{P}-(5,1,4,2,3,3,2,4,1)$ is an integral polytope, thus $K[I_{10,5,2}]$ is Gorenstein.
\end{Example}

We state and prove the main theorem of this paper.
\begin{Theorem}\label{mainresult}
The $t$-spread Veronese algebra, $K[I_{n,d,t}]$, is Gorenstein if and only if $n\in \{(d-1)t+1, (d-1)t+2, dt, dt+1, dt+d\}.$
\end{Theorem}
\begin{proof}
\indent If $n=dt+k$ with $2\leq k\leq d-1$ and $n\geq (t+1)d+1$, then, by Theorem \ref{n1} and Theorem \ref{n2}, $K[I_{n,d,t}]$ is not Gorenstein. Hence, it remains to study the cases when $n\in \{(d-1)t+1, (d-1)t+2, dt, dt+1, dt+d\}.$\\
\indent If $n=(d-1)t+1$, then $K[I_{n,d,t}]$ is a polynomial ring, thus it is Gorenstein. If $n=(d-1)t+2$, by Theorem \ref{polyring}, $K[I_{n,d,t}]$ is Gorenstein. If $n=dt$, by Theorem \ref{dt}, we obtain the same conclusion.\\
\indent Let $\underline{n=dt+1}$. In our hypothesis, by Theorem \ref{dimalg}, $\dim K[I_{n,d,t}]=dt+1$, thus $\dim \mathcal{P}=dt$. Using similar arguments as in Theorem \ref{n1},
\begin{center}
$\mathcal{P}=\{(a_1, a_2, \dots, a_{n-1})\in \mathbb{R}^{n-1}: a_i\geq 0, 1\leq i\leq n-1, a_i+a_{i+1}+\dots+a_{i+t-1}\leq 1, 1\leq i\leq n-t, a_1+a_2+\dots+ a_{n-t}\geq d-1\}.$
\end{center}
We show that the smallest integer $\delta\geq 1$ such that $\delta (\mathcal{P}-\partial \mathcal{P})$ contains lattice points is $t+d$. Assume that there are interior lattice points at lower levels than $t+d$. It is enough to see that there are no interior lattice points at level $t+d-1$. The interior of the dilated polytope is
\begin{center}
$(t+d-1)(\mathcal{P}- \partial \mathcal{P})=\{(a_1, a_2, \dots, a_{n-1})\in \mathbb{R}^{n-1}: a_i> 0, 1\leq i\leq n-1, a_i+a_{i+1}+\dots+a_{i+t-1}< t+d-1, 1\leq i\leq n-t, a_1+a_2+\dots+ a_{n-t}> (d-1)(t+d-1)\}.$
\end{center}
In this case, for each lattice point $(a_1, a_2, \dots, a_{n-1})\in (t+d-1)(\mathcal{P}- \partial \mathcal{P})$, we have $a_i+a_{i+1}+\dots+a_{i+t-1}\leq t+d-2$, for any $1\leq i\leq n-t$, thus $a_1+a_2+\dots+ a_{(d-1)t}\leq (t+d-2)(d-1)$. But $a_1+a_2+\dots+ a_{n-t}\geq (d-1)(t+d)+1$, thus we obtain
\begin{center}
$(d-1)(t+d-1)+1\leq \sum_{i=1}^{n-t}a_i\leq (t+d-2)(d-1)+a_{(d-1)t+1},$
\end{center}
hence, $a_{(d-1)t+1}\geq d$. But, since $a_{(d-2)t+2}+a_{(d-2)t+3}+\dots + a_{(d-1)t+1}\leq t+d-2$, we obtain $a_{(d-2)t+2}+\dots+a_{(d-1)t}\leq t-2$, which is the sum of $t-1$ terms and each $a_{(d-2)t+j}>1$, for any $2\leq j\leq t$. We show that $(t+d)(\mathcal{P}-\partial \mathcal{P})$ contains only one lattice point. The interior of the dilated polytope is
\begin{center}
$(t+d)(\mathcal{P}- \partial \mathcal{P})=\{(a_1, a_2, \dots, a_{n-1})\in \mathbb{R}^{n-1}: a_i> 0, 1\leq i\leq n-1, a_i+a_{i+1}+\dots+a_{i+t-1}< t+d, 1\leq i\leq n-t, a_1+a_2+\dots+ a_{n-t}> (d-1)(t+d)\}.$
\end{center}
Let $(x_1, x_2, \dots, x_{n-1})\in (t+d)(\mathcal{P}- \partial \mathcal{P})\cap \mathbb{Z}^{n-1}$. Thus $x_1+x_2+\dots+ x_{(d-1)t+1}\geq (d-1)(t+d)+1$.
\begin{center}
{\it Claim:} $x_{kt+1}\geq d$, for any $0\leq k\leq d-1$.
\end{center}
Since $x_i+x_{i+1}+\dots + x_{i+t-1}\leq t+d-1$, for any $1\leq i\leq k-1$ and for any $k\leq i\leq d-2$, we obtain
$$ x_1+x_2+\dots +x_{kt}+x_{kt+2}+\dots +x_{(d-1)t+1}\leq (t+d-1)(d-1).$$
Hence,
$$ (d-1)(t+d)+1\leq \sum_{i=1}^{n-t}x_i\leq (d-1)(t+d-1)+x_{kt+1},$$
thus $x_{kt+1}\geq d$, for any $0\leq k\leq d-1$, as we claimed.\\
\indent But, $x_{kt+1}+x_{kt+2}+\dots+ x_{(k+1)t}\leq d+t-1$ and $x_{kt+1}\geq d, x_{kt+j}\geq 1$, for any $2\leq j\leq t$, thus $x_{kt+1}+x_{kt+2}+\dots+ x_{(k+1)t}= d+t-1$. The equality holds if and only if, for any $0\leq k\leq d-1$, $x_{kt+1}=d$ and $x_{kt+j}=1$, for any $2\leq j\leq t$. Therefore, $\alpha=(x_1, x_2, \dots, x_{n-1})$, where

\[
x_{j}= \left\{
\begin{array}{ll}
      d, & j=kt+1, 0\leq k\leq d-1\\
      1 & j=kt+l, 0\leq k\leq d-1, 0\leq l\leq t-1, l\neq 1.
\end{array}
\right.
\]
 is the unique interior lattice point in $(t+d)\mathcal{P}$. Using Theorem \ref{stdtype}, we show that $K[I_{n,d,t}]$ is Gorenstein. Let us compute $\mathcal{Q}=(t+d)\mathcal{P}-\alpha$.
 \begin{center}
$\mathcal{Q}=\{(y_1, \dots, y_{n-1})\in \mathbb{R}^{n-1}: y_i\geq -d, i=kt+1, 0\leq k\leq d-1 , y_i\geq -1, i=kt+l, 0\leq k\leq d-1, 0\leq l\leq t-1, l\neq 1, y_i+y_{i+1}+\dots + y_{i+t-1}\leq 1, 1\leq i\leq n-t, y_1+ \dots + y_{n-t}\geq -1\}.$
\end{center}
In fact, we show that the dual $\mathcal{Q}^{*}$ of $\mathcal{Q}$ is an integral polytope, by showing that the independent hyperplanes which determine the facets of $\mathcal{Q}$ are
\begin{equation*}
\begin{split}
 & y_i= -1, i=kt+l, 0\leq k\leq d-1, 0\leq l\leq t-1, l\neq 1,\\
 &  y_i+y_{i+1}+\dots + y_{i+t-1}= 1, 1\leq i\leq n-t,\\
 & y_1+ \dots + y_{n-t}= -1.
\end{split}
\end{equation*}
Thus, we need to show that all the hyperplanes $y_i= -d, i=kt+1, 0\leq k\leq d-1$ are redundant. Let $0\leq k\leq d-1$. Since $y_i+y_{i+1}+\dots+y_{i+t-1}\leq 1$, for any $1\leq i\leq k-1$ and $y_{it+2}+\dots +y_{(i+1)t+1}\leq 1$, for any $k\leq i\leq d-2$, we obtain
\begin{center}
$y_1+y_2+\dots + y_{(k-1)t+1}+ \dots+ y_{kt}+y_{kt+2}+\dots+y_{n-t}\leq k-1+[d-1-(l-1)]=d-1,$
 \end{center}
 and, since $y_1+y_2+\dots+ y_{(d-1)t+1}\geq -1$, we obtain $y_i\geq -d, i=kt+1$, for any $0\leq k\leq d-1$. Thus, since the supporting hyperplanes of the polytope $\mathcal{Q}$ have integral coefficients, we conclude that $\mathcal{Q}$ is an integral polytope. Hence, by Theorem \ref{stdtype}, $K[I_{n,d,t}]$ is Gorenstein.\\

\indent Let $\underline{n=dt+d}$. In our hypothesis, by Theorem \ref{dimalg}, $\dim K[I_{n,d,t}]=dt+d$, thus $\dim \mathcal{P}=dt+d-1$. We have,
\begin{center}
$\mathcal{P}=\{(a_1, a_2, \dots, a_{n-1})\in \mathbb{R}^{n-1}: a_i\geq 0, 1\leq i\leq n-1, a_i+a_{i+1}+\dots+a_{i+t-1}\leq 1, 1\leq i\leq n-t, a_1+a_2+\dots+ a_{n-t}\geq d-1\}.$
\end{center}
\indent We show that there are no interior lattice points at lower levels than $t+1$. It is enough to see that there are no interior lattice points at level $t$. Let $(a_1, a_2, \dots, a_{n-1})\in t(\mathcal{P}-\partial \mathcal{P})\cap \mathbb{Z}^{n-1}$. We have
\begin{center}
$t(\mathcal{P}-\partial \mathcal{P})=\{(a_1, a_2, \dots, a_{n-1})\in \mathbb{R}^{n-1}: a_i> 0, 1\leq i\leq n-1, a_i+a_{i+1}+\dots+a_{i+t-1}<t, 1\leq i\leq n-t, a_1+a_2+\dots+ a_{n-t}> (d-1)t\}.$
\end{center}
Since each $a_i >0$, for any $1\leq i\leq n-1$, we obtain $t>a_i+a_{i+1}+\dots+a_{i+t-1}\geq t$, which is a contradiction. Thus, there are no interior lattice points in the dilated polytope at lower levels than $t+1$. We show that $(t+1)(\mathcal{P}-\partial \mathcal{P})$ contains only one interior lattice point which has all the coordinates equal to $1$. The interior of the dilated polytope is
\begin{center}
$(t+1)(\mathcal{P}-\partial \mathcal{P})=\{(a_1, a_2, \dots, a_{n-1})\in \mathbb{R}^{n-1}: a_i> 0, 1\leq i\leq n-1, a_i+a_{i+1}+\dots+a_{i+t-1}<t+1, 1\leq i\leq n-t, a_1+a_2+\dots+ a_{n-t}> (d-1)(t+1)\}.$
\end{center}
 We know that, for each lattice point $(a_1, \dots, a_{n-1})\in (t+1)(\mathcal{P}-\partial \mathcal{P})$, we have $a_i+\dots +a_{i+t-1}\leq t$, for any $1\leq i\leq n-t$ and, since $a_j\geq 1$, for any $i\leq j\leq i+t-1$, we obtain $a_i+\dots+a_{i+t-1}\geq t$, thus we have equality which implies that $a_j=1$, for any $1\leq j\leq n-1$. Hence, $(1,1,\dots, 1)\in \mathbb{Z}^{n-1}$ is the unique interior lattice point in the dilated polytope $(t+1)\mathcal{P}$.\\
  \indent Let us consider $\mathcal{Q}=(t+1)\mathcal{P}-(1,1,\dots, 1)$. We will show that $K[\mathcal{P}]$ is Gorenstein by using Theorem \ref{stdtype}. In fact, we show that the dual $\mathcal{Q}^{*}$ of
 \begin{center}
$\mathcal{Q}=\{(y_1, \dots, y_{n-1})\in \mathbb{R}^{n-1}: y_i\geq 0, 1\leq i\leq n-1, y_i+y_{i+1}+\dots + y_{i+t-1}\leq 0, 1\leq i\leq n-t, y_1+ \dots + y_{n-t}\geq -1\}.$
\end{center}
is an integral polytope.
As the hyperplanes $\sum_{j=i}^{i+t-1}y_j=0$, for any $1\leq i\leq n-t$, and $\sum_{i=1}^{n-t}y_i=-1$, have integral coefficients, it follows that $\mathcal{Q}$ is an integral polytope. Thus, by Theorem \ref{stdtype}, we conclude that $K[I_{n,d,t}]$ is Gorenstein.
\end{proof}

\begin{Example}
Let $n=11$, $d=3$ and $t=4$. In this case, $\delta=5$ and in the dilated polytope
\begin{center}
 $5(\mathcal{P}-\partial \mathcal{P})=\{(a_1,\dots, a_{10}) \in \mathbb{R}^{10}: a_i>0, 1\leq i \leq10, i\neq 4, 8, a_9<a_5<a_1, a_9+a_{10}< a_5+a_6< a_1+a_2<5, a_4=a_8=0\}$
 \end{center}
 there is a unique interior lattice point, namely $(3,1,1,0,2,1,2,0,1,1)$. Let us compute the polytope $\mathcal{Q}=5\mathcal{P}-(3,1,1,0,2,1,2,0,1,1)$. Then
 \begin{center}
 $\mathcal{Q}=\{(y_1, y_2, \dots, y_{10})\in \mathbb{R}^{10}:y_i>-1, 1\leq i\leq 10, i\neq 4,8, y_9-y_5<1, y_5-y_1<1, y_9+y_10-y_5-y_6<1, y_5+y_6-y_1-y_2<1, y_1+y_2<1\}.$
 \end{center}
 Thus, the dual polytope of $\mathcal{Q}$ is integral. Therefore, $K[I_{11,3,4}]$ is Gorenstein.
\end{Example}

\begin{Example}
Let $n=10$, $d=3$ and $t=3$. In this case, $\delta=6$ and in the dilated polytope
\begin{center}
 $6(\mathcal{P}-\partial \mathcal{P})=\{(a_1,\dots, a_{9}) \in \mathbb{R}^{9}: a_i>0, 1\leq i \leq9,  a_1+a_2+a_3<6, a_2+a_3+a_4<6, a_3+a_4+a_5<6, a_4+a_5+a_6<6, a_5+a_6+a_7<6, a_6+a_7+a_8<6, a_7+a_8+a_9<6, a_1+a_2+\dots+a_7>12\}$
 \end{center}
 there is a unique interior lattice point, namely $(3,1,1,3,1,1,3,1,1)$. Let us compute the polytope $\mathcal{Q}=6\mathcal{P}-(3,1,1,3,1,1,3,1,1)$. Then
 \begin{center}
 $\mathcal{Q}=\{(y_1, y_2, \dots, y_{9})\in \mathbb{R}^{9}:y_2\geq -1, y_3\geq -1, y_5\geq -1, y_6\geq -1, y_8\geq -1, y_9\geq -1, y_1+y_2+y_3\leq 1, y_2+y_3+y_4\leq 1, y_3+y_4+y_5\leq 1, y_4+y_5+y_6\leq 1, y_5+y_6+y_7\leq 1, y_6+y_7+y_8\leq 1, y_7+y_8+y_9\leq 1, y_1+y_2+\dots+y_7\geq -1\}.$
 \end{center}
 Thus, the dual polytope of $\mathcal{Q}$ is integral. Therefore, $K[I_{10,3,3}]$ is Gorenstein.

\end{Example}

\begin{Example}
Let $n=8$, $d=2$ and $t=3$. In this case, $\delta=4$ and in the dilated polytope
\begin{center}
$4(\mathcal{P}-\partial \mathcal{P})=\{(a_1,\dots, a_{7}) \in \mathbb{R}^{7}: a_i>0, 1\leq i\leq 7, a_1+a_2+a_3<4, a_2+a_3+a_4<4, a_3+a_4+a_5<4, a_4+a_5+a_6<4, a_5+a_6+a_7<4, a_1+a_2+a_3+a_4+a_5>4\}$
\end{center}
there is a unique interior lattice point, $(1,1,\dots, 1)$. Let us compute the polytope $\mathcal{Q}=4\mathcal{P}-(1,1,\dots,1)$. Then
\begin{center}
$\mathcal{Q}=\{(y_1, y_2, \dots, y_7)\in \mathbb{R}^7: y_i>-1, 1\leq i\leq7, y_1+y_2+y_3<1, y_3+y_4+y_5<1, y_4+y_5+y_6<1, y_5+y_6+y_7<1, y_1+y_2+y_3+y_4+y_5>-1\}$.
\end{center}
Thus, the dual polytope $\mathcal{Q}^{*}$ is integral. Therefore, $K[I_{8,2,3}]$ is Gorenstein.
\end{Example}

\section*{Acknowledgement}
This paper was written while the author visited the Department of Mathematics of the University of Duisburg-Essen in July, $2018$. The author wants to express her gratitude to Professor J\"urgen Herzog for his entire support and guidance and to Professor Viviana Ene for useful discussions.
The author thanks European Mathematical Society and Doctoral School in Mathematics of the University of Bucharest for the financial support provided and gratefully acknowledges the use of the computer algebra system Macaulay2 (\cite{mac2}) for experiments.

\end{document}